\documentclass[12pt]{article}
\usepackage{epsfig}
\usepackage{enumerate}

\usepackage[mathscr]{eucal}
\usepackage{amssymb}
\usepackage{amsmath}

\usepackage{graphicx}
\usepackage{color}

\newcommand\C{{\cal C}}

\renewcommand\S{{\cal S}}

\newcommand\K{{\cal K}}

\newcommand\PG{{\rm{PG}}}
\newcommand\PGL{{\mbox{PGL}}}

\newcommand\GF{{\rm{GF}}}

\newcommand\D{{\cal D}}

\renewcommand{\P}{\mathcal{P}}

\newcommand\st{:}
\newcommand\A{{\cal A}}

\newcommand\Ooo{{\cal O}}

\newcommand{\linfty}{\ell_{\infty}}
\newcommand{\li}{\ell_{\infty}}

\renewcommand{\H}{{\mathcal H}}

\newcommand{\PGO}{\mbox{PGO}}

\newcommand\sinfty{{\Sigma_\infty}}
\newcommand\si{{\Sigma_\infty}}

\renewcommand\S{{\cal S}}

\newcommand\cpt{$\C$-point}
\newcommand\cpl{$\C$-plane}
\newcommand\cln{$\C$-line}
\newcommand\ppt{$\infty$-point}
\newcommand\spt{$\star$-point}
\newcommand\sln{$\star$-line}
\newcommand\pln{$\infty$-line}

\newcommand\BB{Bruck Bose representation}

\newcommand\gfqp{\GF(q)\cup\{\infty\}}

\newcommand\nn{{2^n}}
\newcommand\cplus{\C\cup\{P,N\}}

\newcommand{\x}{s}

\newcommand\eo{(1,0,0,0,0)}
\newcommand\eii{(0,0,1,0,0)}

\newcommand\eti{(t^{2^n},1,0,0,0)}
\newcommand\etii{(0,0,t,1,0)}
\newcommand\Ptu{\left(\frac{\x t^\nn}{(t+u)^\nn},\frac{\x}{(t+u)^\nn},\frac{\x t}{t+u},\frac{\x}{t+u},1\right)}

\newcommand{\Label}{\label}

\newtheorem{theorem}{Theorem}[section]
\newtheorem{lemma}[theorem]{Lemma}

\newtheorem{corollary}[theorem]{Corollary}

\newtheorem{conjecture}{Conjecture}

\newenvironment{proof}{\noindent{\bf Proof}\hspace{0.5em}}
    { \null  \hfill $\square$ \par}

\newcommand{\R}{{\cal R}}

\begin{document}
\title{A characterisation of translation ovals in finite even order planes}

\author{S.G. Barwick and Wen-Ai Jackson\\
\date{}
School of Mathematics, University of Adelaide\\
Adelaide 5005, Australia
}

\maketitle

\begin{abstract}
In this article we consider a set $\C$ of points in $\PG(4,q)$, $q$ even, satisfying certain combinatorial properties with respect to the planes of $\PG(4,q)$. We show that there is a regular spread in the hyperplane at infinity, such that in the corresponding Bruck-Bose plane $\PG(2,q^2)$, the points corresponding to $\C$ form a translation hyperoval, and conversely. 
\end{abstract}


\section{Introduction}

In this article we first consider a non-degenerate conic in $\PG(2,q^2)$, $q$ even. We look at the corresponding point set in the Bruck-Bose representation in $\PG(4,q)$, and study its combinatorial properties (details of the Bruck-Bose representation are given in Section~\ref
{sect-background}). Some properties of this set were investigated in \cite{barwcaps}. In this article we are interested in combinatorial properties relating to planes of $\PG(4,q)$. We consider a set of points in $\PG(4,q)$ satisfying certain of these combinatorial properties and find that the points correspond to a translation oval in the Bruck-Bose plane $\PG(2,q^2)$. 

In \cite{conicqodd}, the case when $q$ is odd is considered, and we show that given a set of points in $\PG(4,q)$ satisfying the following combinatorial properties, we can reconstruct the conic in $\PG(2,q^2)$. 
We use the following terminology  in $\PG(4,q)$: if the  hyperplane at infinity is denoted $\si$, then we call the points of $\PG(4,q)\setminus\si$ \textit{affine points}.

\begin{theorem}{\rm \cite{conicqodd}}\Label{mainthmqodd} Let $\sinfty$ be the hyperplane at infinity in $\PG(4,q)$,
  $q\ge7$, $q$ odd.
Let $\C$ be a set of $q^2$ affine points, called $\C$-points,  and suppose there exists a set of 
planes called \textit{\cpl s} satisfying the following properties:
\begin{enumerate}
\item Each $\C$-plane\ meets $\C$ in a $q$-arc. 
\item  Any two distinct $\C$-points lie in a unique  \cpl.
\item The affine points of $\PG(4,q)$ are of three types: points of $\C$;
  points on no \cpl; and points on
  exactly two \cpl s.  
  \item If a plane
  meets $\C$ in more than four points, it is a \cpl.
\end{enumerate}
Then there exists a unique spread $\S$ in $\sinfty$ so that in the Bruck-Bose translation plane $\P(\S)$, the $\C$-points  form a $q^2$-arc of $\P(\S)$.
Moreover, the spread $\S$ is regular, and so $\P(\S)\cong\PG(2,q^2)$, and
the $q^2$-arc can be completed to a conic of $\PG(2,q^2)$. 
\end{theorem}

The case when $q$ is even is more complex. The combinatorial properties only allow us to reconstruct a translation oval in $\PG(2,q^2)$. The main result of this article is the following theorem.

\begin{theorem}\Label{main-theorem}
Consider $\PG(4,q)$, $q$ even, $q>2$, with the hyperplane at infinity denoted by $\si$. Let $\C$ be a set of $q^2$ affine points, called \emph{\cpt s} and consider a set of planes called \emph{\cpl s} which satisfies the following:
\begin{itemize}
\item[{\rm (A1)}]\Label{ass-arc} Each \cpl\ meets $\C$ in a $q$-arc.  
\item[{\rm (A2)}] \Label{ass-uniq} Any two distinct \cpt s lie in a unique \cpl.
\item[{\rm (A3)}]\Label{ass-partition} The affine points that are not in $\C$ lie on exactly one \cpl.
\item[{\rm (A4)}]\Label{ass-nothreeplane} Every plane which meets $\C$ in at least three points either meets $\C$ in exactly four points or is a \cpl.
\end{itemize}
Then there exists a regular spread $\S$ in $\si$ such that in the Bruck-Bose plane $\P(\S)\cong\PG(2,q^2)$, the $\C$-points, together with two extra points on $\li$, form a translation hyperoval of $\PG(2,q^2)$.
\end{theorem}

We begin in Section~\ref{sect-background} with the necessary background material on the Bruck-Bose representation.
In Section~\ref{sect-conic-props} we investigate combinatorial properties of conics and translation ovals in $\PG(2,q^2)$, $q$ even, and show that they satisfy properties (A1-4) of Theorem~\ref{main-theorem}. 
The rest of the article is devoted to proving Theorem~\ref{main-theorem}.
In Section~\ref{open} we discuss further aspects of the problem, and avenues for further work.

\section{The Bruck-Bose representation}\label{sect-background}\label{sect:BB}

We use the Bruck-Bose representation of $\PG(2,q^2)$ in $\PG(4,q)$ introduced in \cite{andr54,bruc64,bruc66}. See \cite{barw08} for more details of this representation. 
Let $\sinfty$ be the hyperplane at infinity of $\PG(4,q)$ and let $\S$ be a spread of
$\sinfty$.
Call the points of $\PG(4,q)\setminus\sinfty$  {\em
  affine points} and the points in $\si$ {\em infinite points}.
The lines and planes of $\PG(4,q)$ that are not contained in $\sinfty$ are called {\em
  affine lines} and {\em affine planes} respectively.

Consider the incidence structure whose {\em points} are the
affine points of $\PG(4,q)$, whose {\em lines} are the affine planes of
PG$(4,q)$ which contain a line of the  spread $\S$, and where {\it incidence\/} is inclusion.
This incidence structure is an affine plane, and can be
completed to a translation plane denoted $\P(\S)$ 
by adjoining the line at infinity $\linfty$ whose points are the
elements of the spread $\S$. 
The translation plane $\P(\S)$ is the Desarguesian plane $\PG(2,q^2)$ if and only if the spread $\S$ is regular.
Note that the affine planes of  $\PG(4,q)$ that do not contain a line of $\S$ correspond to 
Baer subplanes of $\PG(2,q^2)$ secant to $\li$.

We introduce the coordinate notation that we will use in the Bruck-Bose representation of $\PG(2,q^2)$ in $\PG(4,q)$ (so $\S$ is a regular spread). See \cite{barw08} for more details of this coordinatisation. In $\PG(2,q^2)$, points have coordinates $(x,y,z)$, $x,y,z\in\GF(q^2)$ and the line at infinity has equation $z=0$. In $\PG(4,q)$, points have coordinates $(x_0,\ldots,x_4)$, $x_i\in\GF(q)$, and we let the hyperplane at infinity $\si$ have equation $x_4=0$. 
Let $\tau$ be a primitive element in $\GF(q^2)$ with primitive polynomial $x^2-t_1x-t_0$ over $\GF(q)$. Let $\alpha,\beta\in\GF(q^2)$, then we can uniquely write $\alpha=a_0+a_1\tau$, $\beta=b_0+b_1\tau$ for $a_i,b_i\in\GF(q)$. The Bruck-Bose map takes a point $P=(\alpha,\beta,1)\in\PG(2,q^2)\setminus\li$ to the point $P=(a_0,a_1,b_0,b_1,1)\in\PG(4,q)\setminus\si$. In $\PG(2,q^2)$, $\li$ has points $\{P_\infty=(0,1,0)\}\cup\{P_\delta=(1,\delta,0)\st \delta\in\GF(q^2)\}$. If $\delta=d_0+d_1\tau$, for $d_0,d_1\in\GF(q)$, then in $\PG(4,q)$, these points correspond to lines $p_\infty, p_\delta$ of the regular spread $\S$ where
\begin{eqnarray*}
p_\infty&=&\langle(0,0,1,0,0),(0,0,0,1,0)\rangle,\\
p_\delta&=&\langle(1,0,d_0,d_1,0),(0,1,t_0d_1,d_0+t_1d_1,0)\rangle.
\end{eqnarray*}

\section{Properties of conics and translation ovals}\Label{sect-conic-props}

Let $\overline\C$ be a non-degenerate conic in $\PG(2,q^2)$, $q$ even, that meets $\li$ in a point $P_\infty$, so $\overline\C$ has nucleus $N\in\li$. Let $\C=\overline\C\setminus\{P_\infty\}$. We use the term  \textit{\cpt s} for (affine) points of $\PG(2,q^2)$ in $\C$. If $\alpha$ is a Baer subplane secant to $\li$, then $\alpha$ meets $\overline\C$ in either a subconic, or in at most four points; so $\alpha$ meets $\C$ in either a $q$-arc or in at most four points. Baer subplanes secant to $\li$ that meet $\C$ in a $q$-arc are called \textit{\cpl s}. The following properties about \cpl s and \cpt s is straightforward to prove, and is also a special case of Theorem~\ref{oval-cplanes-props}.

\begin{lemma}\Label{conic-cplanes-props}
 Let $\overline\C=\C\cup\{P_\infty\}$ be a conic  in $\PG(2,q^2)$, $q$ even, that meets $\li$ in a point $P_\infty$. Define $\C$-points to be points of $\C$, and $\C$-planes to be Baer subplanes secant to $\li$ that meet $\C$ in a $q$-arc. Then the following hold.
\begin{enumerate}
\item\Label{C1}  Two distinct \cpt s lie in a unique \cpl. 
\item\Label{C2} Every affine point not in $\C$ lies in a unique \cpl.
\item\Label{C3}  If a Baer subplane secant to $\li$ contains at least three \cpt s, then it either contains exactly four \cpt s, or is a \cpl.
 \end{enumerate}
\end{lemma}

Our aim in this article is to work in the Bruck-Bose representation in $\PG(4,q)$, and consider a set of \cpt s and \cpl s that satisfy the properties given in Lemma~\ref{conic-cplanes-props}, and see if we can reconstruct the conic. In fact, it turns out that these geometric properties do not uniquely reconstruct the conic, rather, they determine a translation oval. So we first show that a translation oval satisfies the required geometric properties. 

Recall that a translation oval in $\PG(2,2^h)$  is  projectively equivalent to 
$\overline\Ooo(2^n)=\{ (t^{2^n},t,1)\st  t\in\gfqp\}$ where $(n,h)=1$. Note that $\overline\Ooo(2^n)$ has nucleus $N=(0,1,0)$, and $\overline\Ooo\cup\{N\}$ is a translation hyperoval. Further, if $n=1$, then $\overline\Ooo(2)$ is a non-degenerate conic.

\begin{theorem}\Label{oval-cplanes-props}
Let $\overline\Ooo(2^n)=\{ (t^{2^n},t,1)\st  t\in\gfqp\}$ be a translation oval in $\PG(2,q^2)$, $q=2^h$, $(n,h)=1$, where $\li$ has equation $z=0$. Define \cpt s to be points of $\Ooo=\overline\Ooo(2^n)\setminus\{(1,0,0)\}$. Define \cpl s to be Baer subplanes secant to $\li$ that contain $q$ points of $\Ooo$. Then the following hold.
\begin{enumerate}
\item Each \cpl\ meets $\Ooo$ in a $q$-arc.  
\item Any two distinct \cpt s lie in a unique \cpl.
\item Every affine point is either a \cpt, or lies in exactly one \cpl.
\item Every Baer subplane secant to $\li$ which meets $\Ooo$ in at least three points, either meets $\Ooo$ in exactly four points, or is a \cpl.
\end{enumerate}
\end{theorem}

\begin{proof} Note that part 1 holds trivially. We begin the proof by
calculating the group $G$ of homographies of $\PG(2,q^2)$ fixing $\Ooo$ and $\li$. Let $\sigma$ 
be a homography that fixes $\Ooo$ and $\li$, so $\sigma$ has matrix
\[
M=\begin{pmatrix}
a &b&c\\d&e&f\\g&h&i
\end{pmatrix}, \quad a,\ldots,i\in\GF(q^2),
\]
and $\sigma(x)=Mx$, where $x$ is the column vector representing the coordinates of a point.
As $\sigma$ fixes $\li$ and $\Ooo$, it fixes $P_\infty=(1,0,0)$ and the nucleus $N=(0,1,0)$. Hence $b=h=d=g=0$. 
An affine point $(x, y, z)$ of $\PG(2,q^2)$ belongs to $\Ooo$ if and only if $xz=y^{2^n}$, so the image of the point $(t^{2^n},t,1)$ is on $\Ooo$ if its co-ordinates satisfy $xz=y^{2^n}$. This gives 
\begin{eqnarray}\label{m-eqn}
M=\begin{pmatrix}
e^{2^n}&0&f^{2^n}\\0&e&f\\0&0&1\end{pmatrix},\quad e,f\in\GF(q^2).
\end{eqnarray}
  The determinant of $M$ is $1\cdot e\cdot e^{2^n}$ so we require $e\ne 0$, but there are no restrictions on $f$.  Thus there are $q^2(q^2-1)$ matrices $M$, so $|G|= q^2(q^2-1)$. 
  
Let $\alpha$ be the Baer subplane $\PG(2,q)$, then the subgroup of $G$ fixing $\alpha$ has matrices with form given in (\ref{m-eqn}), but with  $e,f\in\GF(q)$, hence this subgroup has size $q(q-1)$. Thus by the orbit stabilizer theorem, the orbit of $\alpha$ is of size $q(q+1)$.  
Now $\alpha$ meets $\Ooo$ in a $q$-arc, so is a \cpl, so there are at least $q(q+1)$ \cpl s. 
Now we look at the subgroup $H$ of $G$ that stabilizes the \cpt\  $P=(0,0,1)$. The group $H$ consists of homographies with matrices with form given in (\ref{m-eqn}),  with $f=0$, so $|H|=q^2-1$. Hence by the orbit stabilizer theorem, $P$ has orbit in $G$ of size $q^2$, that is, $G$ is transitive on the $q^2$ \cpt s. Next consider the subgroup $I$ of $H$ fixing the \cpt\  $Q=(1,1,1)$. The group $I$ consists of homographies with matrices with form given in (\ref{m-eqn}), with $f=0$ and $e=1$, so $|I|=1$. Hence by the orbit stabilizer theorem, $G$ is 2-transitive on the \cpt s.

We now show that part 2 holds. Let $A,B$ be any two \cpt s. 
Any $\C$-plane\ containing $A$ and $B$ is a Baer subplane that contains the line $\li$ and a $q$-arc of $\Ooo$, and so contains $N$ and $P_\infty$, the unique completion of the $q$-arc to a $(q+2)$-arc. The four points $A,B,P_\infty,N$ form a quadrangle, and so lie in a unique Baer subplane $\alpha$.  Hence there is at most one \cpl \ through the two \cpt s $A,B$. 
Now count in two ways the triples $(A,B,\beta)$, where $A$ and $B$ are distinct \cpt s in the \cpl\ $\beta$.  Firstly we have $q^2(q^2-1)$ pairs $(A,B)$ with at most one \cpl\ containing them.  Thus the number of triples is at most $q^2(q^2-1)$ with equality if and only if every pair is on a unique \cpl.  Alternatively, we calculated in the first paragraph that there are at least $q(q+1)$  \cpl s $\beta$. Each of these \cpl s contains $q$ \cpt s, so there are at least $q(q+1)\times q(q-1)=q^2(q^2-1)$ triples.  
Hence there are exactly $q^2(q^2-1)$ triples, exactly $q(q+1)$ \cpl s, and any two distinct points lie on exactly one \cpl. Hence part 2 holds.

For part 3, consider the affine point $P= (0, 1, 1)$, it is on the line joining two points  $(1, 0, 0)$ and $(1, 1, 1)$ of the oval $\overline\Ooo$, and and so is not a \cpt.  The image of $P$ under the homography $\sigma$ with matrix $M$ given in (\ref{m-eqn}) is 
$P'=MP=(f^{2^h},e+f,1)$. This is equal to $P$ if and only if $f=0$ and $e=1$.
Hence the only element of $G$ fixing $P$ is the identity, so by the orbit stabilizer theorem, $G$ is transitive on the $q^4-q^2$ affine non-\cpt s of $\PG(2,q^2)$.  As each of the $q^2+q$ \cpl s contains $q^2-q$ affine non-\cpt s, it follows that every affine non-\cpt\ is on exactly one \cpl\ and part 3 is proved.

For part 4, 
let $\alpha$ be a Baer subplane secant to $\li$ that contains three \cpt s $A,B,C$.  As the group $G$ is 2-transitive on the \cpt s, without loss of generality, we may assume the points are $A=(0, 0, 1)$, $B=(1, 1, 1)$ and $C=(t^{2^n}, t, 1)$, where $t\in\GF(q^2)\setminus\{ 0,1\}$. So $\alpha$ contains the three points $X=AB\cap\li=(1, 1, 0)$, $Y=AC\cap\li=(t^{2^n}, t, 0)$ and $Z=BC\cap\li=(t^{2^n}+1, t+1, 0)$ on $\li$. As $\alpha$ is a subplane, it follows that the point $D=BY\cap XC=((t+1)^{2^n}, t+1, 1)$ is in $\alpha$.  Note that $D$ also lies on the line $ZA$. However, $D$ is a \cpt\ which is not equal to $A,B$ or $C$, as $t\ne 0,1$.  Thus if $\alpha$ contains at least three \cpt s, then $\alpha$ contains at least four \cpt s. 

We want to show that either $\alpha$ contains exactly four \cpt s, or $\alpha$ is a \cpl. To do this we find the coordinates of the remaining points of $\alpha$.
The homography $\rho$ with matrix 
$$K=\begin{pmatrix}t^{2^n}+1&t^{2^n}&0\\t+1&t&0\\ 0&0&1 \end{pmatrix}$$
maps $(1, 0, 0),\ (0, 1, 0),\ (0, 0, 1),\ (1, 1, 1)$  to the points $Z,Y,A,B$ respectively. Hence $\rho$ maps the Baer subplane $\PG(2,q)$ to $\alpha$.
By part 2, there is a unique $\C$-plane\ through $A,B$, namely $\PG(2,q)$. 
Note that $\alpha=\PG(2,q)$ if and only if $t\in\GF(q)$ (since $C\in\PG(2,q)$ if and only if $t\in\GF(q)$). Suppose $t\notin\GF(q)$, so $\alpha$ is not a \cpl. We need to show that $\alpha$ contains no further \cpt s. Note that $\rho$ maps points of $\li$ to points of $\li$, so the affine points of $\alpha$ are the image under $\rho$ of affine points of $\PG(2,q)$. An affine point $P=(x,y,1)$, $x,y\in\GF(q)$ of $\PG(2,q)$ maps under $\rho$ to the affine point $P'=KP=
(x(t^{2^n}+1)+yt^{2^n}, \ x(t+1)+yt, \ 1)$ of $\alpha$. So the affine points of $\alpha$ are the points with coordinates of form $P'$ for $x,y\in\GF(q)$. Now $P'$ is a \cpt\ 
 if and only if $x(t^{2^n}+1)+yt^{2^n}=(x(t+1)+yt)^{2^n}$. Rearranging gives $t^{2^n}(x+x^{2^n}+y+y^{2^n})+(x+x^{2^n})=0$. As $t\not\in\GF(q)$, we have $t^{2^n}\notin\GF(q)$ as $(n,h)=1$, and so it follows that $x+x^{2^n}+y+y^{2^n}=0$ and $x+x^{2^n}=0$.  The second equation implies that $x\in\GF(2)$, as $(n,h)=1$. Hence $x=0$ or $1$, and then the first equation implies that $y=0$ or $1$.  Thus the only points of $\alpha$ which are  \cpt s are $\rho(0,0,1)=K(0, 0, 1)^t=A$, $\rho(0, 1, 1)=C$, $\rho(1, 0, 1)=D$ and $\rho(1, 1, 1)=B$.  
 Hence if $\alpha$ is not a \cpl, then it contains exactly four \cpt s. 
That is, if a plane contains more than four \cpt s, it is a \cpl, completing the proof of part 4.
\end{proof}

We briefly discuss the interpretation of these properties in the Bruck-Bose representation of $\PG(2,q^2)$ in $\PG(4,q)$. Recall that Baer subplanes of $\PG(2,q^2)$ secant to $\li$ correspond exactly to affine planes of $\PG(4,q)$ that do not contain a line of the spread. So when we consider \cpl s in $\PG(4,q)$, we are interested in affine planes of $\PG(4,q)$. If $\pi$ is a Baer subplane of $\PG(2,q^2)$ secant to $\li$, and $\K$ is a $k$-arc of $\pi$, then in $\PG(4,q)$, $\K$ is a $k$-arc in the affine plane corresponding to $\pi$. Moreover, if $\K$ is a conic in the Baer subplane $\pi$ of $\PG(2,q^2)$, then $\K$ corresponds to a conic in $\PG(4,q)$ (see \cite{quin02}).

\section{Proof of Theorem~\ref{main-theorem}}

For the remainder of this article, we will assume that $q$ is even, and suppose that in $\PG(4,q)$ we have a set of $q^2$ (affine) \cpt s, and a set of \cpl s that satisfy the conditions (A1-4) of Theorem~\ref{main-theorem}.

The proof of  Theorem~\ref{main-theorem} proceeds as follows. In Section~\ref{section-cplanes} we investigate \cpl s, and show that the \cpl s each meet $\si$ in one of $q+1$ {\em $\C$-lines} which are pairwise skew. In Section~\ref{section-cline}, we study the $\C$-lines and show that there are two disjoint lines $t_N$, $t_\infty$ in $\si$  that meet every $\C$-line. In Section~\ref{section-coord} we use the Klein correspondence to coordinatise the $\C$-lines. The remaining proof of Theorem~\ref{main-theorem} in Section~\ref{main-proof}
begins by coordinatising the \cpt s. Then we construct a regular spread $\S$ containing the two special lines $t_N$, $t_\infty$, so that in the Bruck-Bose plane $\P(\S)\cong\PG(2,q^2)$, the points corresponding to $\C$ and $t_N$ and $t_\infty$ form a translation hyperoval.

\subsection{Properties of  \cpl s}\Label{section-cplanes}

In this section we begin by showing that the \cpt s and \cpl s form an affine plane, and then investigate the parallel classes of this affine plane.

\begin{lemma}\Label{affine-plane}
Consider the incidence structure $\A$ with {\em points} the \cpt s, and {\em lines} the \cpl s, and natural incidence. Then $\A$ is 
 an affine plane  of order $q$.
\end{lemma}
\begin{proof}  By assumptions (A1) and (A2), $\A$ is a 2-$(q^2,q,1)$ design, and hence is an affine plane of order $q$. 
\end{proof}

\begin{lemma}\Label{arc}
No three \cpt s are collinear.
\end{lemma}

\begin{proof}
Suppose three \cpt s $A,B,C$ are collinear. By (A2) there is a unique \cpl\ containing $A,B$, this $\C$-plane\ contains the line $AB$, and so contains the point $C$, which contradicts (A1).  Hence is it not possible for three \cpt s to be collinear.
\end{proof}

Consider two distinct \cpl s in $\PG(4,q)$; they either meet in a line $\ell$ or in a point $P$.  
So
there are five possibilities: 
(a) they meet in an infinite line;
(b) they meet in an affine line;
(c) they meet in an infinite point;
(d) they meet in an affine point $P$ that is a \cpt; and
(e) they meet in an affine  point $P$ that is not a \cpt.
Note that as $q>2$, by (A3) cases (b) and (e)
cannot occur. 
We now show  that case (c)
cannot occur either.



\begin{lemma}\Label{no-single-int}
Two \cpl s cannot meet exactly in a point of $\si$.
\end{lemma}
\begin{proof}
Note that if two \cpl s $\alpha$ and $\beta$ are in the same parallel class of the affine plane $\A$, then they have no affine points in common, and so must meet in a line or a point of $\si$ (that is, case (a) or (c)). If $\alpha$ and $\beta$ are in different parallel classes of $\A$, then as case (b) and (e) cannot occur, they must meet in a point of $\C$.

Suppose two \cpl s $\alpha$ and $\beta$ meet exactly in a point $X\in\si$, so $\alpha$ and $\beta$ are in the same parallel class of $\A$. Let $A$ be any \cpt\ of $\alpha$ and let $B,C$ be two \cpt s of $\beta$ such that the line $BC$ does not contain $X$ (this is possible as $q>2$). Consider the 3-space $\Sigma=\langle A,B,C,X\rangle$, note that $\beta\subset\Sigma$. Consider the plane $\pi=\langle A,B,C\rangle$, it contains three \cpt s and hence by (A4) contains a fourth \cpt\ $D$.  Now $D$ does not belong to $\alpha$ as if so, the two lines $AD$ and $BC$ on the plane $\pi$ must meet, and must do so at $\alpha\cap\beta=X$, a contradiction as $BC$ does not contain $X$. Further, $D$ does not belong to $\beta$ as $\pi$ meets $\beta$ in a line which already contains two \cpt s (and no three \cpt s are collinear by Lemma~\ref{arc}).  Now consider the two \cpt s $A$ and $D$, by (A2) there is a unique \cpl\ $\gamma$ containing $A,D$.  As $\gamma$ meets $\alpha$ in the \cpt\  $A$, and $\alpha$ and $\beta$ are in the same parallel class, $\gamma$ meets $\beta$ in a \cpt\ $E$, say. Now $A,D\in\pi\subset \Sigma$ and $E\in\beta\subset\Sigma$ so $\gamma\subset\Sigma$.  So the 3-space $\Sigma$ contains two \cpl s $\beta$ and $\gamma$. However, $\gamma,\beta$ are in different parallel classes of $\A$, and so meet exactly in the \cpt\ $E$. Hence $\gamma,\beta$ cannot lie in a 3-space. So no two \cpl s meet exactly at a point in $\si$.
\end{proof}

Hence, \cpl s can meet in one of two ways, and so in the affine plane $\A$ defined in Lemma~\ref{affine-plane}, we have the following properties about parallel classes. 
\begin{lemma}\Label{parallel-class} 
\begin{enumerate}
\item Two \cpl s lie in the same parallel class of $\A$ if and only if their intersection is a line in $\si$.
\item Two \cpl s lie in distinct parallel classes of $\A$ if and only if their intersection is exactly a point of $\C$.
\end{enumerate}
\end{lemma}

As a direct consequence, the parallel classes of $\A$ allow us to define a set of lines in $\si$ that lie on \cpl s.

\begin{corollary}\Label{def-cline}
 Each parallel class of \cpl s meets $\si$ in a common line, called a $\C$-line. There are $q+1$ distinct $\C$-lines, and they are pairwise skew.
\end{corollary}

\subsection{Properties of \cln s}\Label{section-cline}

In this section we investigate properties of the $\C$-lines in $\si$ defined in Corollary~\ref{def-cline}. In particular, we show that we can construct two lines $t_N, t_\infty$ that meet every $\C$-line.

Let $\alpha$ be a \cpl, then by (A1) the set ${\mathcal Q}=\alpha\cap\C$ is a $q$-arc. As $q>2$, ${\mathcal Q}$ is contained in a unique oval of $\alpha$, and hence a unique hyperoval ${\mathcal Q}^+$, see \cite{tall57}. The two points of ${\mathcal Q}^+\setminus {\mathcal Q}$ are called {\em \ppt s} (we use this name as we will show in Corollary~\ref{inf-pts-in-si} that these two points lie in the hyperplane at infinity $\si$). The line of the $\C$-plane $\alpha$ joining the two \ppt s is called the {\em \pln}\ of $\alpha$ (again, we will show that this line is in $\si$). 
Note that the \pln\ contains no \cpt s, otherwise this would contradict ${\mathcal Q}^+$ being a hyperoval. 
We also need notation for the $q-1$ points of the \pln\ of $\alpha$ that are not \ppt s; these points are called {\em \spt s}. It is straightforward to verify the following lemma.

\begin{lemma}\Label{arc-complete}
Let $\alpha$ be a \cpl, and let ${\mathcal Q}=\alpha\cap\C$.
Let $X$ be of a point of $\alpha$, with $X\notin{\mathcal Q}$.  Then
\begin{enumerate}
\item If $X$ is a \ppt,\ then the lines of $\alpha$ through $X$ are $q$ $1$-secants and one $0$-secant of ${\mathcal Q}$.
\item If $X$ is a \spt,\ then the lines of $\alpha$ through $X$ are the \pln, and $\frac q2$ $2$-secants, and $\frac q2$ $0$-secants of ${\mathcal Q}$. 
\item If $X$ is not on the \pln, then the lines of $\alpha$ through $X$ are $(\frac q2 -1)$ $2$-secants, two $1$-secants (through the \ppt s) and $\frac q2$ $0$-secants of ${\mathcal Q}$.
\end{enumerate}
\end{lemma}

We now investigate the \ppt s and will show they all lie in $\si$.

\begin{lemma}\Label{plus-pts-coincide}
If two \cpl s meet in a line, then their \ppt s coincide, and lie in $\si$.
\end{lemma}
\begin{proof}
Let $\alpha,\beta$ be two \cpl s that meet in a line $t$. By Lemma~\ref{parallel-class}, $t$ is in $\si$.
 Suppose that $\alpha$ and $\beta$ do not share \ppt s.  
 We construct a 1-secant $\ell$ of $\C$ in $\alpha$ and a 2-secant $m$ in $\beta$ as follows. 
 Firstly, if at least one plane, $\alpha$ say, has \ppt s in $\si$, then let $X$ be an \ppt\ of $\alpha$ which is not an \ppt\ of $\beta$, and let $\ell $ be a 1-secant of $\C\cap\alpha$ through $X$. Secondly, if neither $\alpha$ nor $\beta$ has an \ppt\ in $\si$, let $X$ be any point of $\alpha\cap\beta$ not on the \pln s of $\alpha$ or $\beta$.  By Lemma~\ref{arc-complete}, we can join $X$ to an \ppt\ of $\alpha$ to form a 1-secant $\ell$.  In both cases, as $q>2$, again by Lemma~\ref{arc-complete}, there exists a 2-secant $m$ in $\beta$ through $X$.  Let $A$ be the \cpt\ on $\ell$ and let $B$ and $C$ be the \cpt s on $m$.  

Consider the plane  $\pi=\langle A,B,C\rangle$; it meets $\C$ in three points, so by assumption (A4) contains a further \cpt\ $D$. Note that $\pi\cap\beta=m$ is a 2-secant of $\C$, and $\pi\cap\alpha=\ell$ is a 1-secant, so $D$ does not belong to $\alpha$ or $\beta$. By (A2), there is a \cpl\ $\gamma$ containing $A$ and $D$. As $\gamma$ meets $\alpha$ in a \cpt, and  $\alpha$ and $\beta$ belong to the same parallel class,  $\gamma$ also meets $\beta$ in a \cpt\  $E$. 
Consider the 3-space $\Sigma=\langle \pi,E\rangle$. Now $A,B,E$ are three non-collinear points in $\Sigma\cap\gamma$, so $\gamma\subset\Sigma$. Similarly, $B,C,E$ are three non-collinear points in $\Sigma\cap\beta$, so $\beta\subset\Sigma$. But $\beta,\gamma$ have a common \cpt\  $E$, so by Lemma~\ref{parallel-class} they meet in exactly $E$ and so cannot lie in a common 3-space, a contradiction. Thus $\alpha$ and $\beta$ must have the same \ppt s, and so their \ppt s lie in $\si$, and hence their \pln s lie in $\si$. 
\end{proof}

By Corollary~\ref{def-cline}, \cpl s in the same parallel class meet in a line of $\si$, hence it follows that all the \ppt s lie in $\si$, and the \pln s are $\C$-lines. 

\begin{corollary}\Label{inf-pts-in-si} All the \ppt s and \pln s lie in $\Sigma_\infty$.
\end{corollary}

We now investigate lines of $\si$ that meet two $\C$-lines.

\begin{lemma}\Label{three-spts}
Any line which meets two \cln s in \spt s meets exactly one other \cln, and meets it  in a \spt.
\end{lemma}

\begin{proof}
Let $\ell$, $m$ be $\C$-lines (so by Corollary~\ref{inf-pts-in-si}, they are $\infty$-lines).
Let $X\in\ell$, $Y\in m$ be \spt s.
By Lemma~\ref{arc-complete}, there is a line through $X$ containing two \cpt s $A,B$.
By Corollary~\ref{def-cline}, $m$ defines a parallel class, and so $\langle m,A\rangle$ is a $\C$-plane in this parallel class.
So
by Lemma~\ref{arc-complete}, the line $AY$ in the $\C$-plane $\langle m,A\rangle$    contains another \cpt, $C$ say.
By (A2), $B,C$ lie in a $\C$-plane\ $\gamma$, further $\gamma$ does not belong to the parallel class defined by either $\ell$ or $m$ as $Z=BC\cap\si$ is not in $\ell$ or $m$.
So $\gamma$ meets $\si$ in a $\C$-line $n$ which by Corollary~\ref{def-cline} is disjoint from $\ell$ and $m$.
Consider the plane $\pi=\langle A,B,C\rangle$.
We have $X,Y\in\pi$, so the point $Z$ lies on the line $XY$, that is $Z=XY\cap BC\in n$.
By Corollary~\ref{inf-pts-in-si}, $n$ is the $\infty$-line of $\gamma$, so as $Z$ lies on a 2-secant $BC$ of $\C$, by Lemma~\ref{arc-complete}, $Z$ is a \spt.
By Lemma~\ref{arc-complete}, the line $AZ$ contains a further \cpt, $D$ say, so the plane $\pi$ contains four \cpt s, $A,B,C,D$.
Note that $\pi\cap\si=XY$ is not a $\C$-line, as $\C$-lines are disjoint, so $\pi$ is not a \cpl.

So we have shown that $XY=\pi\cap\si$ meets another $\C$-line $n$ in a \spt\ $Z$.
Suppose that $XY$ meets a further $\C$-line $t$ in a point $W$. Note that $X,Y\in\pi$, so $W\in\pi$. As $B,C,W,A\in\pi$, let $E=BC\cap WA$ and note that $E$ is distinct from $A,B,C,D$. If $E$ is a \cpt, then $\pi $ contains five \cpt s, $A,B,C,D,E$, contradicting (A4) as $\pi$ is not a \cpl. 
If $E$ is not a \cpt, then $E$ is on two \cpl s, namely  $\langle A,t\rangle$ and $\langle B,n\rangle$, this contradicts (A3).
Hence $XY$ cannot meet four $\C$-lines. That is, $XY$ meets exactly one other $\C$-line, and it meets it in a \spt. 
\end{proof}

By Corollary~\ref{inf-pts-in-si}, the \pln s are $\C$-lines, and the \ppt s lie on the $\C$-lines. By Lemma~\ref{plus-pts-coincide}, \cpl s in the same parallel class have the same \ppt s. By Corollary~\ref{def-cline}, there are $q+1$ distinct, pairwise disjoint $\C$-lines. Hence there are $2(q+1)$ distinct \ppt s. We now show that the $2(q+1)$ \ppt s form two disjoint lines in $\si$.

\begin{lemma}\Label{line-ppts}
The $2(q+1)$ \ppt s form two disjoint lines of $\si$. We label these lines $t_N$, $t_\infty$.
\end{lemma}

\begin{proof} In this proof we are interested in the lines of $\si$ that meet three $\C$-lines in \spt s. These lines exist by Lemma~\ref{three-spts}, and in this proof, we call these lines \sln s.
Let $A,B$ be \spt s on distinct $\C$-lines. By Lemma~\ref{three-spts}, $AB$ is a \sln, and contains exactly one further \spt, $C$ say. 
We will count in two ways the number $x$ of ordered triples $(A,B,C)$ where $A,B,C$ are \spt s on a common \sln. Firstly, there are $(q+1)(q-1)$ choices for $A$; then $q(q-1)$ choices for $B$, yielding a unique $C$. So $x=(q+1)(q-1)q(q-1)$.
Secondly, we count the number of \sln s first. 
Let $r_0$ be a \sln, so $r_0$ meets three $\C$-lines $\ell,m,n$. These three $\C$-lines determine a unique regulus $\R$, with opposite regulus $\R'=\{r_0,\ldots,r_q\}$. Each line of $\R'$ meets the three $\C$-lines $\ell,m,n$, so by Lemma~\ref{three-spts}, they meet no further $\C$-line. Further, 
if any of the lines $r_i$ contains two \spt s, then it contains three \spt s, and is a \sln.
The maximum number of \sln s in $\R'$ is $q-1$, which occurs when the \ppt s of $\ell,m,n$ all lie on two lines of $\R'$.
So the total number of \sln s is at most $\binom{q+1}3\times (q-1)$.
The number of ordered triples of \spt s on a \sln\ is $3!$. Hence this count gives $x\leq\binom{q+1}3\times (q-1)3!=(q+1)q(q-1)^2$. 
Thus $x= (q+1)q(q-1)^2$, and for each set of three $\C$-lines, the \ppt s lie on two common lines. Hence all the \ppt s lie on two lines of $\si$. 
\end{proof}

\subsection{Coordinatising the $\C$-lines}\Label{section-coord}

In Sections~\ref{section-cplanes} and~\ref{section-cline}, we showed that given a set of \cpt s and \cpl s in $\PG(4,q)$ satisfying (A1-4), we can construct from them $q+1$ $\C$-lines, and two special lines $t_N$ and $t_\infty$  in $\si$. We now show that without loss of generality, we can coordinatise these special lines of $\si$ as outlined in Theorem~\ref{cline-coord}.

To determine this coordinatisation, we use the Klein correspondence between lines of $\si\cong\PG(3,q)$ and points of the Klein quadric $\H_5$ in $\PG(5,q)$.
The line $PQ$ of $\PG(3,q)$ with $P=(p_0,p_1,p_2,p_3)$, $Q=(q_0,q_1,q_2,q_3)$ maps to the point $(\ell_{01},\ell_{02},\ell_{03},\ell_{12},\ell_{31},\ell_{23})$ of $\H_5$ in $\PG(5,q)$ where $\ell_{ij}=p_iq_j-p_jq_i$. The Klein quadric $\H_5$ has equation $x_0x_5+x_1x_4+x_2x_3=0$. 
For background on the Klein correspondence, see \cite{hirs85}.

\begin{theorem}\Label{cline-coord}
In $\PG(4,q)$ we can without loss of generality coordinatise the lines $t_N$, $t_\infty$ and the $\C$-lines $m_t$, $t\in\GF(q)\cup\{\infty\}$ as follows:
\begin{eqnarray*}
t_N&=&\langle (0,0,1,0,0),\ (0,0,0,1,0)\rangle,\quad\\
t_\infty&=&\langle (1,0,0,0,0),\ (0,1,0,0,0)\rangle,\\
m_t&=&\langle(t^{2^n},1,0,0,0),\ (0,0,t,1,0)\rangle,\quad t\in\GF(q)\cup\{\infty\}.
\end{eqnarray*}
\end{theorem}

\begin{proof} In this proof, we use the Klein correspondence between the lines of $\si\cong\PG(3,q)$ and the points of the Klein quadric $\H_5$ in $\PG(5,q)$. In $\PG(4,q)$, a point $(x_0,x_1,x_2,x_3,x_4)$ in $\si$ has last coordinate $x_4=0$, so we write the coordinates of points in $\si$ as $(x_0,x_1,x_2,x_3)$.
Let $A=(1,0,0,0)$, $B=(0,1,0,0)$, $C=(0,0,1,0)$, $D=(0,0,0,1)$ be points of $\si$. Without loss of generality we can let $t_\infty=AB$ and $t_N=CD$. Consider the lines $\ell_1=AC,\ \ell_2=AD,\ \ell_3=BC,\ \ell_4=BD$. In $\si$, the $\C$-lines and $\ell_1,\ldots,\ell_4$ all meet $t_N$ and $t_\infty$, so they lie in the hyperbolic congruence consisting of the $(q+1)^2$ transversals of $t_N,t_\infty$.
This hyperbolic congruence maps under the Klein correspondence to a hyperbolic quadric $\H_3$ contained in a 3-space $\Sigma$ in $\PG(5,q)$ (see \cite{hirs85}).  Hence in $\PG(5,q)$, the $\C$-lines correspond to a set $\K$ of $q+1$ points that lie on $\H_3$. By Corollary~\ref{def-cline}, the $\C$-lines are pairwise skew, hence in $\PG(5,q)$ the points of $\K$ are such that no two lie on a generator line of $\H_3$. Hence there is exactly one point of $\K$ on each generator of $\H_3$.

Now consider the result of Lemma~\ref{three-spts}. Let $m_1,m_2$ be two $\C$-lines of $\si$, and let $\ell$ be a line that meets $m_1,m_2$ in \spt s. Then $\ell$ meets exactly one more $\C$-line, $m_3$ say. In $\PG(5,q)$, $\ell$ corresponds to a point $L$ not in $\Sigma$ (the 3-space containing $\H_3$), and the line $m_i$ corresponds to a point $M_i$ in $\K$. Now Lemma~\ref{three-spts} says that if $LM_1$, $LM_2$ are lines of the Klein quadric $\H_5$, then there is exactly one more line $LM_3$ of $\H_5$ through $L$ that contains a point $M_3$ of $\K$. Now  $\langle L,M_1,M_2,M_3\rangle$ is a 3-space as $M_1,M_2,M_3$ are not on a generator of $\H_3$. So $\langle L,M_1,M_2,M_3\rangle$ meets $\H_5$ in a conic cone, and so meets $\Sigma$ in a conic. Conversely, a conic of $\H_3$ will lie on such a conic cone of $\H_5$. Hence any conic of $\H_3$ contains at most three points of $\K$. Hence each plane of $\Sigma$ meets $\K$ in at most three points. Thus $\K$ is a $(q+1)$-arc of $\Sigma$.

The images of the lines $\ell_1,\ell_2,\ell_3,\ell_4$ of $\si$ under the Klein correspondence are $L_1=(0,1,0,0,0,0)$, 
$L_2=(0,0,1,0,0,0)$, 
$L_3=(0,0,0,1,0,0)$, 
$L_4=(0,0,0,0,1,0)$. 
These four points lie in the 3-space $\Sigma$ determined by the equations $x_0=0,\ x_5=0$. If we use $(x_1,x_2,x_3,x_4)$ for the coordinates of $\Sigma$, then $\Sigma$ meets $\H_5$ in the hyperbolic quadric of equation $\H_3: x_1x_4+x_2x_3=0$. 

In \cite[Section 21.3]{hirs85}, a line $\ell$ through a point $P$ of $\K$ is called a \textit{special unisecant} of $\K$ if every plane through $\ell$  contains at most one other point of $\K$.  Further, it is shown that there are exactly two special unisecants through each point $P$ of $\K$, and the special unisecants are the generators of a quadric $\H(\K)$.  We note that since there is exactly one point of $\K$ on each generator $g$ of $\H_3$, then the $q+1$ planes through $g$ defined by $g$ and one of $q+1$ lines in the opposite regulus of $\H_3$ contains at most one  point of $\K$.  This argument shows that $\H(\K)=\H_3$.

By \cite[Theorem 21.3.15]{hirs85}, the $(q+1)$-arc $\K$ is projectively equivalent to $\C(2^n)=\{(t^{2^n+1},t^{2^n},t,1)\st t\in\GF(q)\cup\{\infty\}\}$ for some $n$ with $(n,h)=1$. Further, $\C(2^n)$ determines through it special unisecants the hyperbolic quadric $\H(\C(2^n))$ which can be shown to equal $\H_3$ (with equation $x_1x_4+x_2x_3=0$).  Thus the homography $\sigma$ mapping $\K$ to $\C(2^n)$ fixes $\H_3$. 
The homographies of $\PGO_+(6,q)$ that fix the Klein quadric are in 1-1 correspondence with the homographies of $\PGL(4,q)$ that map lines of $\si$ to lines of $\si$. Hence $\sigma$  
corresponds to a homography of $\si$ that fixes the hyperbolic congruence consisting of all transversals of $t_N,t_\infty$. So without loss of generality, we can coordinatise the $\C$-lines to be the lines that correspond to points of $\C(2^n)$. Straightforward calculations show that the line $m_t=\langle(t^{2^n},1,0,0),(0,0,t,1)\rangle$ for $t\in\GF(q)\cup\{\infty\}$ corresponds under the Klein map to the point $M_t=(0,t^{2^n+1},t^{2^n},t,1,0)$ which lies in $\C(2^n)$. Hence we can without loss of generality coordinatise the $\C$-lines and $t_N,t_\infty$ as stated in the theorem.
\end{proof}

\subsection{Proof of the Main Theorem}\Label{main-proof}

In this section we complete the proof of Theorem~\ref{main-theorem}.

\textbf{Proof of Theorem~\ref{main-theorem}\ } Suppose we have a set of \cpt s and \cpl s in $\PG(4,q)$ that satisfy (A1-4). Then we have shown that in the hyperplane at infinity $\si$, there are two special lines $t_N$, $t_\infty$ and  $q+1$ $\C$-lines $m_t$, $t\in\GF(q)\cup\{\infty\}$. In Theorem~\ref{cline-coord} we coordinatised these lines. We now coordinatise the \cpt s. 

Without loss of generality we can choose $A=(0,0,0,0,1)$ to be a $\C$ point. 
The plane spanned by $A$ and the $\C$-line $m_\infty$ is a \cpl, that is
$$\alpha_{(A,\infty)}=\langle A,m_\infty\rangle=\langle (0,0,0,0,1), \eo,\eii\rangle$$
is a \cpl. Every line through $A$ in this plane either meets $t_N$ or $t_\infty$, or contains exactly one further \cpt.  
Points in  $\PG(4,q)$ have coordinates $(x_0,x_1,x_2,x_3,x_4)$. Points in this plane  $\alpha=\alpha_{(A,\infty)}$  have $x_1=x_3=0$, so we use the coordinates 
$(x_0,x_2,x_4)$ to represent points of $\alpha$. With this notation, $A=(0,0,1)$, and consider the points $R_N=\alpha\cap t_N=(0,1,0)$, $R_\infty=\alpha\cap t_\infty=(1,0,0)$. Consider the point $Q=(1,1,0)$ on the line $R_NR_\infty$. The line $AQ$ consists of points $A,Q$ and $(s,s,1)$, $s\in\GF(q)\setminus\{0\}$. As $Q\neq R_N,R_\infty$,  there exists some $s\in\GF(q)\setminus\{0\}$ such that $B=(s,s,1)$ is a \cpt. Hence in $\PG(4,q)$, we can take $B=(s,0,s,0,1)$ to be a \cpt, where $s$ is some nonzero element of $\GF(q)$.

We can use $A$ and $B$ to calculate the coordinates of the remaining \cpt s and \cpl s. Firstly, the plane spanned by $A$ and any $\C$-line $m_t$ is a \cpl, similarly for $B$. This give us the \cpl s
\[
\begin{array}{ccccccc}
\alpha_{(A,t)}&\!\!=\!\!&\langle A,m_t\rangle&\!\!=\!\!&\langle (0,0,0,0,1),\eti,\etii\rangle,&\!\!t\in\GF(q)\cup\{\infty\},\\
\alpha_{(B,u)}&\!\!=\!\!&\langle B,m_u\rangle&\!\!=\!\!&\langle (s,0,s,0,1),(u^{2^n},1,0,0,0),(0,0,u,1,0)\rangle&\!\!u\in\GF(q)\cup\{\infty\}.
\end{array}
\]
Note that
\begin{eqnarray*}
\alpha_{(A,\infty)}&=&\langle (0,0,0,0,1),(1,0,0,0,0),(0,0,1,0,0)\rangle,\\
\alpha_{(B,\infty)}&=&\langle (s,0,s,0,1),(1,0,0,0,0),(0,0,1,0,0)\rangle.
\end{eqnarray*}
By (A2), $A,B$ lie on a unique \cpl, which is  $\alpha_{(A,\infty)}=\alpha_{(B,\infty)}$. So this gives
coordinates of $(2q+1)$ \cpl s.
By Lemma~\ref{affine-plane}, two distinct \cpl s meet in a unique \cpt. 
The coordinates of the $q(q-1)$ \cpt s $P_{(A,t),(B,u)}=\alpha_{(A,t)}\cap\alpha_{(B,u)}$
for $t,u\in\GF(q),\ t\neq u$ are
\begin{eqnarray*}
P_{(A,t),(B,u)}&=&\Ptu.
\end{eqnarray*}
The remaining $q-1$ \cpt s are in $\alpha_{(A,\infty)}$, we calculate their coordinates next. As $B\in\alpha_{(A,\infty)}$, we first need to calculate some \cpl s that meet $\alpha_{(A,\infty)}$.
Let $u\in\GF(q)\setminus\{0\}$ and consider the \cpt\ 
$$
P_{(A,0),(B,u)}=
{\left(0,\ \frac{\x}{u^\nn},\ 0,\ \frac{\x}{u},\ 1\right)}
$$ which lies in the $\C$-plane\ $\alpha_{(A,0)}$. 
Consider the $q$ \cpl s $\alpha_{(P,t)}=\langle P_{(A,0),(B,u)},m_t\rangle$ for $t\in\GF(q)\setminus\{0\}$. 
These $q$ \cpl s each meet $\alpha_{(A,\infty)}$ in a \cpt. Hence the $q-1$ \cpt s in $\alpha_{(A,\infty)}$ distinct from $A$ are
 $$\alpha_{(A,\infty)}\cap\alpha_{(P,t)}={\left(\frac{\x t^\nn}{u^\nn},\ 0,\ \frac{\x t}{u},\ 0,\ 1\right)},$$ $t\in\GF(q)\setminus\{0\}$.
Note that we  only need one parameter here, let $\psi=t/u$, then the $q-1$ \cpt s are $$Q_\psi=(s\psi^{2^n},0,s\psi,0,1),$$
$\psi\in\GF(q)\setminus\{0\}$. Hence we have calculated the coordinates of the $q^2$ \cpt s. 

We show that we can choose a regular spread $\S$ such that  in the corresponding plane $\P(\S)\cong\PG(2,q^2)$ (under the Bruck-Bose map) 
the \cpt s form a translation oval.

Let $\tau$ be a primitive element of $\GF(q^2)$, and let $n<h$ be a positive integer such that $(n,h)=1$. Then we can uniquely write $\tau^{2^n}$ as $\tau^{2^n}=a_0+a_1\tau$ for some $a_0,a_1\in\GF(q)$. Note that as $(n,h)=1$, $\GF(2^n)\cap\GF(2^h)=\GF(2)$, hence $a_1\neq0$.
Now consider the homography $\sigma$ of $\PG(4,q)$ with matrix
$$M=\begin{pmatrix}
1&a_0&0&0&0\\
0&a_1&0&0&0\\
0&0&1&0&0\\
0&0&0&1&0\\
0&0&0&0&1
\end{pmatrix}.$$
The \cpt s 
$P_{(A,t),(B,u)},Q_{\psi}$ are mapped under $\sigma$ to 
\begin{eqnarray*}
P_{(t,u)}&=&\sigma(P_{(A,t),(B,u)})={\left(\frac{\x (t^\nn+a_0)}{(t+u)^\nn},\ a_1\frac{\x}{(t+u)^\nn},\ \frac{\x t}{t+u},\ \frac{\x}{t+u},\ 1\right)}\\
Q_{\psi}&=&\sigma(Q_\psi)=
{\left(\x \psi^\nn,\ 0,\ \x \psi,\ 0,\ 1\right)}.
\end{eqnarray*}
The \cln s $m_\infty$, $m_t$, $t\in\GF(q)$ are mapped under $\sigma$ to
\begin{eqnarray*}
m'_\infty&=&\langle\eo,\eii\rangle=m_\infty,\\
m'_t&=&\langle (t^\nn+a_0,a_1,0,0,0),\etii\rangle.
\end{eqnarray*}
Note that the special lines $t_N,t_\infty$ are fixed by the homography $\sigma$. 

So without loss of generality, we can map the \cpt s, $\C$-lines and \cpl s so that the special lines are 
$t_N=\langle (0,0,1,0,0),\ (0,0,0,1,0)\rangle$, $t_\infty=\langle (1,0,0,0,0),\ (0,1,0,0,0)\rangle$; 
the $\C$-lines are $m'_t=\langle (t^\nn+a_0,a_1,0,0,0),\etii\rangle$, $t\in\GF(q)\cup\{\infty\}$; and the \cpt s are $P_{(t,u)}$, $Q_\psi$, $t,u,\psi\in\GF(q)$, $t\neq u$, $\psi\neq0$. 
We now look at this in the Bruck-Bose plane $\PG(2,q^2)$ using the regular spread $\S$ coordinatised in Section~\ref{sect:BB}. Note that the primitive element $\tau$ used to extend $\GF(q)$ to $\GF(q^2)$ in this coordinatisation is the same as the $\tau$ used to determine $a_0,a_1$ in the matrix $M$ (so $\tau^{2^n}=a_0+a_1\tau$). 
Also note that $t_N$, $t_\infty$ are two lines of this regular spread, since using the notation of Section~\ref{sect:BB} we have $t_N=p_\infty$ and $t_\infty=p_0$.
Using this regular spread $\S$, we look in the related Bruck-Bose plane $\P(\S)\cong\PG(2,q^2)$ and show that the points of $\C$ correspond to a translation oval in $\PG(2,q^2)$.

Under the Bruck-Bose map, in $\PG(2,q^2)$, the 
point $P_{t,u}$ corresponds to the point
\begin{eqnarray*}
P_{t,u}&=&{\left(\frac{\x (t^\nn+a_0)}{(t+u)^\nn}+\frac{\x a_1}{(t+u)^\nn}\tau,\ \frac{\x t}{t+u}+ \frac{\x}{t+u}\tau,\ 1\right)}.
\end{eqnarray*}
Let $$\theta=\theta(t,u)=\frac{t}{t+u}+ \frac{1}{t+u}\tau,$$ then $$\{P_{(t,u)}\st  t,u\in\GF(q), t\neq u\}=\{\left(\x\theta^\nn,\ \x\theta,\ 1\right)\st \theta\in\GF(q^2)\backslash\GF(q)\}.$$
The point $Q_\psi$, $\psi\in\GF(q)\setminus\{0\}$ of $\PG(4,q)$ corresponds in $\PG(2,q^2)$ to the point $Q_\psi=(\x\psi^\nn,\x\psi,1)$.  
The spread lines $t_\infty$ and $t_N$ correspond in $\PG(2,q^2)$ to the points $P_\infty=(1,0,0)$ and $N=(0,1,0)$ respectively on the line at infinity $\li$.
Consider the set of points $\K$ in $\PG(2,q^2)$ consisting of the images of $P_{t,u},Q_{t,u},P_\infty,N$, so $$\K=\{ {\left(\x\theta^\nn,\ \x\theta,\ 1\right)}  \st \theta\in\GF(q^2)\}\cup\{(1,0,0)\}\cup\{(0,1,0)\}.$$
Then $\K$ is the image of the translation oval $\Ooo(2^n)$ with nucleus $N(0,1,0)$ under the 
projectivity with matrix
$$\begin{pmatrix}
\x&0&0\\0&\x&0\\0&0&1
\end{pmatrix}.$$
Hence if we have a set of \cpt s and \cpl s satisfying (A1-4), then they arise from a translation oval in $\PG(2,q^2)$.  This completes the proof of Theorem~\ref{main-theorem}.
\hfill$\square$

\section{Discussion}\Label{open}

In the case when $q$ is odd, (see \cite{conicqodd}) the assumptions of Theorem~\ref{mainthmqodd} allow us to construct a unique spread $\S$, such that in the Bruck-Bose plane $\P(\S)\cong\PG(2,q^2)$ the \cpt s, together with $P_\infty$, form a conic. In the case when $q$ is even, the assumptions of Theorem~\ref{main-theorem} do not lead to a unique spread. We have shown that there exists a regular spread $\S$ in $\si$ such that in the Bruck-Bose plane $\P(\S)\cong\PG(2,q^2)$, the \cpt s together with the point at infinity $P_\infty$ form a translation oval.
It is possible that there are other spreads $\S'$ for which the \cpt s in the Bruck-Bose plane $\P(\S')$ form an arc.
If there is such a spread $\S'$, we show it must contain $t_N$ and $t_\infty$, and that the remaining lines of $\S'$ each meet exactly one $\C$-line.

\begin{lemma}\Label{spread-cond}
 Let $\S'$ be any spread of $\Sigma_\infty$ containing $t_N$ and $t_\infty$.  Then in the Bruck-Bose plane $\P(\S')$, the set $\C\cup\{P_\infty,N\}$ is a hyperoval if and only if each spread line (other than $t_N$ and $t_\infty$) meets exactly one \cln.
\end{lemma}

\begin{proof}
Let $\S'$ be a spread of $\si$ containing $t_N$ and $t_\infty$ such that each spread line (other than $t_N$ and $t_\infty$) meets exactly one \cln.  We show that $\C\cup\{P_\infty,N\}$ is a hyperoval in the Bruck-Bose plane $\P(\S'$). 
If an affine plane $\pi$ of $\PG(4,q)$ meets $\C$ in at least three points, $A,B,C$, then we use a very similar argument to the proof of Lemma~\ref{three-spts} to find a fourth point $\C$-point $D$ such that the diagonal points of the quadrangle $ABCD$ are \spt s. Hence a plane that meets $\C$ in more than two points is either a \cpl\  or  meets three \cln s. 
Thus every plane through a line of $\S'$ (other than $t_N$ and $t_\infty$) contains at most two \cpt s. Note that 
the affine planes through $t_N$ and $t_\infty$ all contain exactly one \cpt\ by Lemma~\ref{arc-complete}, as these lines contain only \ppt s.  
Hence in the Bruck-Bose plane $\P(\S')$, every line meets $\C\cup\{N,P_\infty\}$ in at most two points, and so it is a hyperoval.

Conversely, consider a spread  $\S'$ containing $t_N$ and $t_\infty$.  If in the Bruck-Bose plane $\P(\S')$, the set $\C\cup\{P_\infty,N\}$  is a hyperoval, then in $\PG(4,q)$, every affine plane through a spread line contains at most two \cpt s. So if a spread line (other than $t_N$ and $t_\infty$) meets two \cln s, then by Lemma~\ref{three-spts} and its proof, we can construct a plane through the spread line with four \cpt s, contradicting  $\C\cup\{N,P_\infty\}$ being a hyperoval.  Thus each spread line  (other than $t_N$ and $t_\infty$) meets exactly one \cln.
\end{proof}

Note that we can create a spread satisfying Lemma~\ref{spread-cond} by considering derivation. See \cite{barw08} for details on derivation and its representation in the Bruck-Bose correspondence. 
Let $\C$ be a conic of $\PG(2,q^2)$, $q=2^h$, meeting $\li$ in a point $P_\infty$ with nucleus $N\in \li$. Let $h$ be even, and let $\D$ be a derivation set (that is, a Baer subline of $\li$) not containing $N$ or $P_\infty$, and such that $P_\infty$ and $N$ are conjugate points with respect to $\D$. In  \cite{glyn93} it was shown that upon deriving with respect to $\D$, the points of $\C$ correspond to a translation oval in the derived plane $\H(q^2)$ (the Hall plane).
In $\PG(4,q)$, the derivation set $\D$ corresponds to a regulus $\R$ of the regular spread $\S$. Let $\R'$ be the opposite regulus of $\R$. Then the spread $\S'=\S\setminus \R\cup\R'$ corresponds (via the Bruck-Bose map) to the derived plane $\P(\S')=\H(q^2)$. By \cite{glyn93}, the \cpt s with $P_\infty$ form a translation oval in this plane. 
We conject that the only spread $\S'$ in $\si$ for which the \cpt s form a $q^2$-arc in the Bruck-Bose
plane $\P(\S')$  are those obtained from the regular spread $\S$ by reversing one or more disjoint reguli not containing $t_N$ or $t_\infty$.

\end{document}